\documentclass{amsart}
\usepackage{graphicx}
\usepackage{amsrefs}
\usepackage[colorlinks]{hyperref}
\newtheorem{theorem}{Theorem}[section]

\newtheorem{lemma}[theorem]{Lemma}

\theoremstyle{definition}
\newtheorem{definition}[theorem]{Definition}
\newtheorem{remark}[theorem]{Remark}
\newtheorem{example}[theorem]{Example}

\numberwithin{equation}{section}

\begin{document}
\title[]{Partial-dual genus polynomial of graphs}
\author{Zhiyun Cheng}
\address{School of Mathematical Sciences, Beijing Normal University, Beijing 100875, China}
\email{czy@bnu.edu.cn}

\begin{abstract}
Recently, Chmutov introduced the partial duality of ribbon graphs, which can be regarded as a generalization of the classical Euler-Poincaré duality. The partial-dual genus polynomial $^\partial\varepsilon_G(z)$ is an enumeration of the partial duals of $G$ by Euler genus. For an intersection graph derived from a given chord diagram, the partial-dual genus polynomial can be defined by considering the ribbon graph associated to the chord diagram. In this paper, we provide a combinatorial approach to the partial-dual genus polynomial in terms of intersection graphs without referring to chord diagrams. After extending the definition of the partial-dual genus polynomial from intersection graphs to all graphs, we prove that it satisfies the four-term relation of graphs. This provides an answer to a problem proposed by Chmutov in \cite{Chm2023}.
\end{abstract}
\maketitle

\section{Introduction}\label{section1}
As a generalization of the fundamental concept of the Euler-Poincaré dual of an embedded graph, the partial dual of ribbon graphs with respect to a subset of edges was proposed by Chmutov in \cite{Chm2009} to realize the Jones polynomial as a specialization of the Bollobás-Riordan polynomial. Unlike the classical Euler-Poincaré dual, which preserves the genus of the embedded surface, the genus can be changed under a partial dual. For this reason, Gross, Mansour and Tucker introduced the notion of partial-dual genus polynomial $^\partial\varepsilon_G(z)$ in \cite{GMT2020}, which enumerates the partial duals of the ribbon graph $G$ by Euler genus. This polynomial has been intensively studied in the last few years, see \cite{GMT2021,GMT2021+,YJ2022,YJ2022+,CF2021,CL2024}.

Chord diagrams play an important role in Vassiliev's finite type invariant theory. Many knot invariants are determined by the intersection graphs of these chord diagrams. For example, in \cite{BG1996} Bar-Natan and Garoufalidis found that the weight system corresponding to the Conway polynomial can be described in terms of the adjacency matrix of the intersection graph. Later, this result was generalized by Mellor \cite{Mel2003} by describing the weight systems associated to HOMFLY polynomial and Kauffman polynomial in terms of the adjacency matrix of the intersection graph. Recall that both the HOMFLY polynomial and Kauffman polynomial are unchanged under mutations. Knot invariants does not distinguish mutants are described precisely by Chmutov and Lando in \cite{CL2007}. Actually, it was proved that a knot invariant does not distinguish mutants if and only if its weight system only depends on the intersection graphs of chord diagrams. 

For a given chord diagram, after attaching a vertex-disc to the circle and thickening the chords, one obtains a ribbon graph with exactly one vertex-disc. It is worth emphasizing that in order to investigate the partial-dual genus polynomial of ribbon graphs, it is sufficient to focus on this kind of ribbon graphs, say bouquets. Recently, Chmutov proved that the partial-dual genus polynomial considered as a function on chord diagrams satisfies the four-term relation in Vassiliev's finite type invariant theory \cite{Chm2023}. In \cite{YJ2022}, Qi Yan and Xian'an Jin proved that if a ribbon graph $G$ is derived from a chord diagram, then the partial-dual genus polynomial $^\partial\varepsilon_G(z)$ depends only on the intersection graph of the chord diagram. More precisely, in \cite[Theorem 5.2]{YJ2022}, the authors gave a recurrent formula for the partial-dual genus polynomial when the ribbon graph is a bouquet satisfying some special condition. In this paper, we provide a combinatorial approach to the partial-dual genus polynomial in terms of intersection graphs without referring to chord diagrams.

\begin{theorem}\label{theorem1}
Let $D$ be a chord diagram and $G$ the associated ribbon graph, then the partial-dual genus polynomial
\begin{center}
$^\partial\varepsilon_{G}(z)=\sum\limits_{A\subseteq E(G)}z^{\emph{rank}(M_A)+\emph{rank}(M_{A^c})}$.
\end{center}
\end{theorem}

Here $A^c$ denotes the complement of $A$ in $E(G)$. Both the matrices $M_A$ and $M_{A^c}$ can be obtained from the adjacency matrix of the intersection graph of $D$ by deleting some rows and columns. The precise definitions of them can be found in Section \ref{section3}. The equation above can be regarded as the definition of the partial-dual genus polynomial of all graphs, not necessary limited to intersection graphs. In particular, this polynomial satisfies the four-term relation of graphs, which was introduced by Lando in \cite{Lan2000}. The precise definition can be found in Section \ref{section3}.

\begin{theorem}\label{theorem2}
Let $I$ be a simple graph, then the partial-dual genus polynomial satisfies the following four-term relation
\begin{center}
$^\partial\varepsilon_{I}(z)-^\partial\varepsilon_{I_{ab}'}(z)-^\partial\varepsilon_{\tilde{I}_{ab}}(z)+^\partial\varepsilon_{\tilde{I}_{ab}'}(z)=0$,
\end{center}
here $a$ and $b$ are two vertices of the graph $I$.
\end{theorem}

The rest of this paper is arranged as follows. In Section \ref{section2}, we recall the definition of the partial-dual genus polynomial and intersection graphs of chord diagrams. Section \ref{section3} is devoted to the proof of Theorem \ref{theorem1} and Theorem \ref{theorem2}. Finally, in Section \ref{section4} we provide a recurrent formula for the partial-dual genus polynomial of graphs, which extends the corresponding result of Qi Yan and Xian'an Jin given in \cite{YJ2022}. The relation between the partial-dual genus polynomial and the newly introduced skew characteristic polynomial is also discussed.

\section{Partial-dual genus polynomial and intersection graph}\label{section2}
Consider a cellular embedding of a graph in a (not necessarily orientable) surface $\Sigma$, a regular neighborhood of the graph in $\Sigma$ is called the associated \emph{ribbon graph} of the embedded graph. In other words, a ribbon graph $G$ is a 2-dimensional manifold represented as the union of two sets, the set $V(G)$ of vertex-discs and the set $E(G)$ of edge-ribbons. The discs in the complement of $G$ in $\Sigma$ form the set of face-discs $F(G)$. The Euler-Poincaré dual of $G$ defines a new ribbon graph $G^*$ in $\Sigma$, which has vertex-discs $V(G^*)=F(G)$, edge-ribbons $E(G^*)=E(G)$ and face-discs $F(G^*)=V(G)$. Since both $G$ and $G^*$ admit a cellular embedding in $\Sigma$, they share the same Euler genus. Recall that the \emph{Euler genus} $\varepsilon(G)$ of a ribbon graph $G$ is equal to twice the genus of $\Sigma$ if $\Sigma$ is orientable, otherwise it is defined to be the crosscap number of $\Sigma$. More precisely, if we use $c(G)$ to denote the number of components of $G$, then the Euler genus $\varepsilon(G)=2c(G)-|V(G)|+|E(G)|-|F(G)|$.

In \cite{Chm2009}, Chmutov generalized the concept of Euler-Poincaré dual to the partial dual with respect to a subset of $E(G)$. Choose a subset $A\subseteq E(G)$, the \emph{partial dual} of $G$ with respect to $A$ is defined as follows. First, consider the spanning surface $F_A$, which consists of $V(G)$ and only the edges in $A$. Then we glue a disc to each boundary component of $F_A$ and remove the interior of all vertex-discs of $V(G)$. Now we obtain the partial dual ribbon graph $G^A$. Notice that $V(G^A)=F(F_A)$, and $E(G^A)=E(G)$, where edges not in $A$ are the same as that in $G$ but edges in $A$ are now reglued. It is not difficult to observe that $(G^A)^A=G$ and if we choose $A=E(G)$, then we have $G^A=G^*$. The reader is referred to \cite{Chm2009} for more information about the partial dual.

As we mentioned above, the Euler genus is preserved under the Euler-Poincaré dual. However, usually it is not invariant under the partial dual. The following definition was introduced by Gross, Mansour and Tucker in \cite{GMT2020}.

\begin{definition}
Let $G$ be a ribbon graph, the partial-dual genus polynomial of $G$ is defined as
\begin{center}
$^\partial\varepsilon_G(z)=\sum\limits_{A\subseteq E(G)}z^{\varepsilon(G^A)}$.
\end{center}
\end{definition}

Notice that here the sum in the definition runs over all the subsets of $E(G)$, together with the fact that $(G^A)^B=G^{A\cup B-A\cap B}$, it is easy to conclude that $^\partial\varepsilon_G(z)=^\partial\varepsilon_{G^A}(z)$ for any subset $A\subseteq E(G)$.

Now we turn to the concept of chord diagram. A \emph{chord diagram} is an oriented circle with finitely many chords inside. The vector space of chord diagrams modulo the four-term relation plays an important role in Vassiliev's finite type invariant theory, see Figure \ref{figure1}. There is a multiplication and a comultiplication on chord diagrams which makes this vector space into a bialgebra. Motivated by this, Lando introduced the so called 4-bialgebra of graphs in \cite{Lan2000}. Later we will revisit this concept in Section \ref{section3}.

\begin{figure}[h]
\centering
\includegraphics{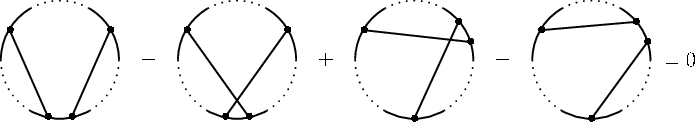}\\
\caption{Four-term relation of chord diagrams}\label{figure1}
\end{figure}

For a given chord diagram $D$, by attaching a disc to the circle and thickening all the chords, one obtains a ribbon graph which contains exactly one vertex-disc. Let us use $G_D$ to denote it. Now we can consider the partial-dual genus polynomial as a function on chord diagrams. Recently, Chmutov showed that the partial-dual genus polynomial on chord diagrams satisfies the four-term relation depicted in Figure \ref{figure1} \cite{Chm2023}. On the other hand, each chord diagram $D$ has an associated intersection graph $I_D$: each chord of $D$ corresponds to a vertex of $I_D$ and two vertices are connected by an edge if and only if the corresponding two chords intersect. See Figure \ref{figure2}.

\begin{figure}[h]
\centering
\includegraphics{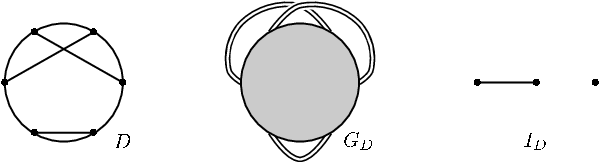}\\
\caption{Chord diagram $D$ and its ribbon graph and intersection graph}\label{figure2}
\end{figure}

It was proved in \cite{YJ2022} that the partial-dual genus polynomial of $G_D$ is completely determined by the intersection graph $I_D$. In other words, two chord diagrams share the same partial-dual genus polynomial if their intersection graphs are isomorphic. When we talk about the partial-dual genus polynomial of a chord diagram $D$, actually we mean the partial-dual genus polynomial of the corresponding ribbon graph. Since it is determined by the intersection polynomial, we can discuss the partial-dual genus polynomial of intersection graphs. This concept will be extended to all graphs in Section \ref{section3}.

\begin{remark}
In \cite{YJ2022}, Qi Yan and Xian'an Jin considered ribbon graphs with exactly one vertex-disc, i.e. bouquets. Untwisted edge and twisted edge are assigned with positive sign and negative sign respectively. Hence the corresponding intersection graph is a vertex-signed graph. However, the ribbon graph derived from a chord diagram is always orientable. Since all the vertices of the intersection graph are positive, we omit the sign in the current paper.
\end{remark}

\section{The proof the main result}\label{section3}
As before, for a given chord diagram $D$, let us use $G$ and $I$ to denote the associated ribbon graph and intersection graph, respectively. Assume the vertex set of $I$ is $\{v_1, \cdots, v_n\}$, the \emph{adjacency matrix} of $I$ is an $n\times n$ matrix $M_I\in \text{M}_n(\mathbb{Z}_2)$, where the entry appears in the $i$th row and $j$th column is equal to 1 if $v_i$ and $v_j$ are adjacent, otherwise it is equal to 0. In particular, all entries on the diagonal are equal to 0. Let $A$ be a subset of the chord set of $D$, then $A$ also can be regarded as a subset of $E(G)$. Let us use $I_A$ to denote the intersection graph corresponding to the chord diagram obtained from $D$ by removing all the chords not in $A$, and use $M_A$ to denote the adjacency matrix of it. Note that the size of $M_A$ is equal to $|A|$.

In order to prove Theorem \ref{theorem1}, we need the following lemma. This result was first proved by Beck in \cite{Bec1977}, later it was extended to the signed case in \cite{Mel2003}.   

\begin{lemma}
Let $D$ be a chord diagram and $G$ the associated ribbon graph, then $|F(G)|=\emph{corank}(M_I)+1$.
\end{lemma}

Now we give the proof of Theorem \ref{theorem1}.

\begin{proof}
Assume $|E(G)|=n$. According to the definition, the partial-dual genus polynomial of $G$ has the form
\begin{center}
$^\partial\varepsilon_{G}(z)=\sum\limits_{A\subseteq E(G)}z^{\varepsilon(G^A)}$.
\end{center}
According to \cite[Corollary 2.3]{GMT2020}, if $G$ is a bouquet, then the Euler genus 
\begin{center}
$\varepsilon(G^A)=\varepsilon(F_A)+\varepsilon(F_{A^c})$.
\end{center}
Recall that the Euler genus
\begin{center}
$\varepsilon(G)=2c(G)-|V(G)|+|E(G)|-|F(G)|$,
\end{center}
it follows that 
\begin{align*}
\varepsilon(F_A)&=2c(F_A)-|V(F_A)|+|E(F_A)|-|F(F_A)|\\
&=2-1+|E(F_A)|-\text{corank}(M_A)-1\\
&=|E(F_{A})|+\text{rank}(M_{A})-|A|\\
&=\text{rank}(M_{A}),
\end{align*}
and
\begin{align*}
\varepsilon(F_{A^c})&=2c(F_{A^c})-|V(F_{A^c})|+|E(F_{A^c})|-|F(F_{A^c})|\\
&=2-1+|E(F_{A^c})|-\text{corank}(M_{A^c})-1\\
&=|E(F_{A^c})|+\text{rank}(M_{A^c})-|A^c|\\
&=\text{rank}(M_{A^c}).
\end{align*}
Therefore, we have
\begin{center}
$\varepsilon(G^A)=\varepsilon(F_A)+\varepsilon(F_{A^c})=\text{rank}(M_{A})+\text{rank}(M_{A^c})$,
\end{center}
and hence
\begin{center}
$^\partial\varepsilon_{G}(z)=\sum\limits_{A\subseteq E(G)}z^{\text{rank}(M_A)+\text{rank}(M_{A^c})}$.
\end{center}
The proof is finished.
\end{proof}

The formula given in Theorem \ref{theorem1} allows us to extend the concept of partial-dual genus polynomial from intersection graphs to all graphs.

\begin{definition}\label{definition}
Let $I$ be a graph and $A$ a subset of the vertex set of $I$, the \emph{partial-dual genus polynomial} of $I$ is defined as
\begin{center}
$^\partial\varepsilon_{I}(z)=\sum\limits_{A\subseteq V(I)}z^{\text{rank}(M_A)+\text{rank}(M_{A^c})}$.
\end{center}
Here $M_A$ and $M_{A^c}$ denote the princepal submatrices of $M_I$ induced by $A$ and $A^c$.
\end{definition}

\begin{remark}
As an analogue of the partial-dual genus polynomial, Qi Yan and Xian'an Jin introduced the concept of twist polynomial for delta-matroids \cite{YJ2022+}. For a given delta-matroid $D=(E, \mathcal{F})$, the \emph{twist polynomial} of $D$ is defined as
\begin{center}
$^\partial w_D(z)=\sum\limits_{A\subseteq E}z^{w(D\ast A)}$,
\end{center}
where $D\ast A$ denotes the twist of $D$ with respect to $A$ and $w$ denotes the width. The reader is referred to \cite{YJ2022+} for the details. It was proved in \cite[Theorem 11]{YJ2022+} that two normal binary delta-matroids share the same twist polynomial if they have the same intersection graph. Similar to Theorem \ref{theorem1}, it is not difficult to find a formula for the twist polynomial in terms of intersection graphs without referring to normal binary delta-matroids.
\end{remark}

\begin{example}
Consider the complete graph $K_n$ on $n$ vertices. Notice that all the entries of adjacency matrix $M_{K_n}$ are equal to 1, except the diagonal, which are all 0. It follows that rank$(M_{K_n})=n$ if $n$ is even, otherwise rank$(M_{K_n})=n-1$. On the other hand, one also notices that the subgraph induced by $A\subseteq V(K_n)$ is the complete graph $K_{|A|}$. Next, we continue our discussion in two cases.
\begin{itemize}
  \item If $n$ is even, then one computes
  \begin{align*}
  ^\partial\varepsilon_{K_n}(z)&=\sum\limits_{A\subseteq V(K_n)}z^{\text{rank}(M_A)+\text{rank}(M_{A^c})}\\
  &=\sum\limits_{|A|=0}^n\binom{n}{|A|}z^{\text{rank}(M_A)+\text{rank}(M_{A^c})}\\
  &=\binom{n}{0}z^n+\binom{n}{1}z^{n-2}+\binom{n}{2}z^n+\cdots+\binom{n}{n-1}z^{n-2}+\binom{n}{n}z^n\\
  &=\sum\limits_{0\leq k\leq \frac{n}{2}}\binom{n}{2k}z^n+\sum\limits_{0\leq k\leq \frac{n-2}{2}}\binom{n}{2k+1}z^{n-2}\\
  &=2^{n-1}z^n+2^{n-1}z^{n-2}.
  \end{align*}
  \item If $n$ is odd, then one computes
  \begin{align*}
  ^\partial\varepsilon_{K_n}(z)&=\sum\limits_{A\subseteq V(K_n)}z^{\text{rank}(M_A)+\text{rank}(M_{A^c})}\\
  &=\sum\limits_{|A|=0}^n\binom{n}{|A|}z^{\text{rank}(M_A)+\text{rank}(M_{A^c})}\\
  &=\sum\limits_{|A|=0}^n\binom{n}{|A|}z^{n-1}\\
  &=2^nz^{n-1}.
  \end{align*}
\end{itemize}
\end{example}

\begin{example}
Consider the complete bipartite graph $K_{m, n}$. Notice that the adjacency matrix $M_{K_{m, n}}$ has the form 
$\begin{pmatrix}
\textbf{0}_{m\times m} & \textbf{1}_{m\times n}\\
\textbf{1}_{n\times m} & \textbf{0}_{n\times n}
\end{pmatrix}$, it follows that rank$(M_{m, n})=2$ if $m\geq1, n\geq1$ and rank$(M_{m, n})=0$ if either $m$ or $n$ equals 0. As a consequence, the partial-dual genus polynomial of $K_{m, n}$ has the form 
\begin{center}
$^\partial\varepsilon_{K_{m, n}}(z)=a_0+a_2z^2+a_4z^4$,
\end{center}
where $a_0, a_2, a_4$ are nonnegative integers. In order to determine $a_0, a_2, a_4$, let us suppose $V(K_{m, n})=\{v_1, \cdots, v_m, w_1, \cdots, w_n\}$, where each pair $v_i, w_j$ are adjacent. Let $A$ be a subset of $V(K_{m, n})$. 
\begin{itemize}
\item Note that $\text{rank}(M_A)+\text{rank}(M_{A^c})=0$ means $\text{rank}(M_A)=\text{rank}(M_{A^c})=0$, which happens if and only if $A=\{v_1, \cdots, v_m\}$ or $\{w_1, \cdots, w_n\}$. It follows that $a_0=2$.
\item For $a_2$, note that $\text{rank}(M_A)+\text{rank}(M_{A^c})=2$ implies that $\text{rank}(M_A)=2$, $\text{rank}(M_{A^c})=0$ or $\text{rank}(M_A)=0$, $\text{rank}(M_{A^c})=2$.
\begin{enumerate}
\item If $\text{rank}(M_A)=2$, $\text{rank}(M_{A^c})=0$, then $A=V(K_{m, n})$ or $\{v_1, \cdots, v_m\}$$\subseteq A$ and $1\leq |A\cap\{w_1, \cdots, w_n\}|\leq n-1$, or $\{w_1, \cdots, w_n\}$$\subseteq A$ and $1\leq|A\cap\{v_1, \cdots, v_m\}|\leq m-1$.
\item If $\text{rank}(M_A)=0$, $\text{rank}(M_{A^c})=2$, then $A=\emptyset$ or $\{v_1, \cdots, v_m\}\cap A=\emptyset$ and $1\leq |A\cap\{w_1, \cdots, w_n\}|\leq n-1$, or $\{w_1, \cdots, w_n\}\cap A=\emptyset$ and $1\leq|A\cap\{v_1, \cdots, v_m\}|\leq m-1$.
\end{enumerate}
As a consequence, $a_2=1+(2^n-2)+(2^m-2)+1+(2^n-2)+(2^m-2)=2^{m+1}+2^{n+1}-6$.
\item Since $a_0+a_2+a_4=2^{m+n}$, it follows that $a_4=2^{m+n}-2^{m+1}-2^{n+1}+4$.
\end{itemize}
Finally, we obtain 
\begin{center}
$^\partial\varepsilon_{K_{m, n}}(z)=2+(2^{m+1}+2^{n+1}-6)z^2+(2^{m+n}-2^{m+1}-2^{n+1}+4)z^4$.
\end{center}
\end{example}

Now we take a quick review of the four-term relation of graphs introduced by Lando in \cite{Lan2000}. Consider a simple graph $I$ and two vertices $a, b\in V(I)$, there are two associated graphs $I_{ab}'$ and $\tilde{I}_{ab}$. The graph $I_{ab}'$ is obtained from $I$ by removing the edge $ab$ if it exists, or by adding it if it does not exist. In other words, the only difference between the graph $I_{ab}'$ and $I$ is the adjacency between the vertices $a$ and $b$. For the other graph $\tilde{I}_{ab}$, suppose all the neighbours (except $a$ if it is) of $b$ are $\{c_1, \cdots, c_k\}$, then $\tilde{I}_{ab}$ is obtained from $I$ by changing the adjacency between $a$ and $c_i$ $(1\leq i\leq k)$. Note that $\tilde{I}_{ab}\neq \tilde{I}_{ba}$ in general, but $\widetilde{I_{ab}'}=(\tilde{I}_{ab})'$. We simply use $\tilde{I}_{ab}'$ to denote the composition of these two operations.

Let $f$ be a graph invariant valued in an abelian group, we say $f$ is a \emph{4-invaraint} if for any simple graph $I$ and any pair of vertices $a, b\in V(I)$, it satisfies the following four-term relation 
\begin{center}
$f(I)-f(I_{ab}')-f(\tilde{I}_{ab})+f(\tilde{I}_{ab}')=0$.
\end{center}
Several well known graph polynomials are 4-invariants, such as $(-1)^{|V(I)|}\chi(I)$, here $\chi(I)$ denotes the chromatic polynomial of $I$. Similar to the finite type invariants, the space of graphs modulo four-term relations can be equipped with a multiplication and a comultiplication to make itself into a bialgebra. The main motivation for introducing the four-term relation for graphs is that, the map from chord diagrams to intersection graphs induces a homomorphism from the set of chord diagrams modulo four-term relation to the set of intersection graphs modulo four-term relation \cite[Theorem 3.3]{Lan2000}. As a consequence, each 4-invariant determines a weight system. In \cite[Problem 2]{Lan2000}, Lando asked are there some other graph invariants that are 4-invariants? Theorem \ref{theorem2}  suggests that the extended partial-dual genus polynomial defined in Definition \ref{definition} provides such an example.

Now we give the proof of Theorem \ref{theorem2}.
\begin{proof}
The main idea of the proof is similar to that of \cite[Proposition 2.8]{Lan2000}. We sketch it here.

Choose two vertices $a, b\in V(I)$, we divide all the subsets of $V(I)$ into the disjoint union of four parts $V_1\cup V_2\cup V_3\cup V_4$, where
\begin{itemize}
  \item $V_1=\{A\subseteq V(I)|a\in A, b\in A^c\}$,
  \item $V_2=\{A\subseteq V(I)|a\in A^c, b\in A\}$,
  \item $V_3=\{A\subseteq V(I)|a\in A, b\in A\}$,
  \item $V_4=\{A\subseteq V(I)|a\in A^c, b\in A^c\}$.
\end{itemize}
Now the partial-dual genus polynomial of $I$ can be rewritten as
\begin{center}
$^\partial\varepsilon_{I}(z)=\sum\limits_{i=1}^4\sum\limits_{A\in V_i}z^{\text{rank}(M_A)+\text{rank}(M_{A^c})}=2\sum\limits_{i=1, 3}\sum\limits_{A\in V_i}z^{\text{rank}(M_A)+\text{rank}(M_{A^c})}$.
\end{center}

We continue the discussion according to the following two cases:
\begin{enumerate}
  \item If $A\in V_1$, i.e. $a\in A$ and $b\in A^c$, then the operation changing the adjacency between $a$ and $b$ has no effect on $I_A$. Therefore the two graphs $I_A(I)$ and $I_A(I_{ab}')$ are exactly the same, here we use $I_A(I)$ and $I_A(I_{ab}')$ to denote the subgraphs obtained from $I$ and $I_{ab}'$ by removing all vertices not in $A$, respectively. Similarly, we also have $I_A(\tilde{I}_{ab})=I_A(\tilde{I}_{ab}')$. It immediately follows that
\begin{center}
$\sum\limits_{A\in V_1}z^{\text{rank}(M_{A}(I))+\text{rank}(M_{A^c}(I))}=\sum\limits_{A\in V_1}z^{\text{rank}(M_{A}(I_{ab}'))+\text{rank}(M_{A^c}(I_{ab}'))}$
\end{center}
and
\begin{center}
$\sum\limits_{A\in V_1}z^{\text{rank}(M_{A}(\tilde{I}_{ab}))+\text{rank}(M_{A^c}(\tilde{I}_{ab}))}=\sum\limits_{A\in V_1}z^{\text{rank}(M_{A}(\tilde{I}_{ab}'))+\text{rank}(M_{A^c}(\tilde{I}_{ab}'))}$.
\end{center}
Here $M_A(I)$ denotes the adjacency matrix of $I_A(I)$ and other notations can be defined similarly.
  \item If $A\in V_3$, i.e. $\{a, b\}\subseteq A$ and hence $a, b\notin A^c$, thus the graphs $I_{A^c}(I)$, $I_{A^c}(I_{ab})$, $I_{A^c}(\tilde{I}_{ab})$, $I_{A^c}(\tilde{I}_{ab}')$ are all the same. Then we have
\begin{center}
$\text{rank}M_{A^c}(I)=\text{rank}M_{A^c}(I_{ab}')=\text{rank}M_{A^c}(\tilde{I}_{ab})=\text{rank}M_{A^c}(\tilde{I}_{ab}')$.
\end{center}
For the subgraphs spanned by the vertices in $A$, without loss of generality, we only consider the case that the vertex $b$ has exactly two neighbours, say $c$ and $d$. We also assume that $\{a, c\}$ are adjacent but $\{a, d\}$ and $\{c, d\}$ are not adjacent. The general case can be proved in a similar manner. Now we notice that the four adjacency matrices look like
\begin{center}
$M_A(I)=\begin{pmatrix}
0 & 0 & 1 & 0 & \ast \\
0 & 0 & 1 & 1 & \textbf{0} \\
1 & 1 & 0 & 0 & \ast \\
0 & 1 & 0 & 0 & \ast \\
\ast & \textbf{0} & \ast & \ast & \ast
\end{pmatrix}, M_A(I_{ab})=\begin{pmatrix}
0 & 1 & 1 & 0 & \ast \\
1 & 0 & 1 & 1 & \textbf{0} \\
1 & 1 & 0 & 0 & \ast \\
0 & 1 & 0 & 0 & \ast \\
\ast & \textbf{0} & \ast & \ast & \ast
\end{pmatrix}$
\end{center}

\begin{center}
$M_A(\tilde{I}_{ab})=\begin{pmatrix}
0 & 0 & 0 & 1 & \ast \\
0 & 0 & 1 & 1 & \textbf{0} \\
0 & 1 & 0 & 0 & \ast \\
1 & 1 & 0 & 0 & \ast \\
\ast & \textbf{0} & \ast & \ast & \ast
\end{pmatrix}, M_A(\tilde{I}_{ab}')=\begin{pmatrix}
0 & 1 & 0 & 1 & \ast \\
1 & 0 & 1 & 1 & \textbf{0} \\
0 & 1 & 0 & 0 & \ast \\
1 & 1 & 0 & 0 & \ast \\
\ast & \textbf{0} & \ast & \ast & \ast
\end{pmatrix}$
\end{center}
Notice that the second row and second column of each matrix contain exactly two nonzero entries, since the vertex $b$ has only two neighbours. By using some elementary linear algebra, one observes that
\begin{center}
$\text{rank}M_A(I)=\text{rank}M_A(\tilde{I}_{ab})$ 
\end{center}
and
\begin{center}
$\text{rank}M_A(I_{ab}')=\text{rank}M_A(\tilde{I}_{ab}')$.
\end{center}
\end{enumerate}

To sum up, we conclude that 
\begin{center}
$^\partial\varepsilon_{I}(z)-^\partial\varepsilon_{I_{ab}'}(z)-^\partial\varepsilon_{\tilde{I}_{ab}}(z)+^\partial\varepsilon_{\tilde{I}_{ab}'}(z)=0$.
\end{center}
\end{proof}

\section{Some discussions}\label{section4}
\subsection{A recurrent formula of the partial-dual genus polynomial}
In \cite[Theorem 5.2]{YJ2022}, Qi Yan and Xian'an Jin found a recurrent formula of the partial-dual genus polynomial for intersection graphs with at least one vertex of degree one. The following result extends this result from intersection graphs to all graphs. Note that the proof of \cite[Theorem 5.2]{YJ2022} is based on the calculation of the partial-dual genus polynomial of bouquets, the proof we give here is purely algebraic.

\begin{theorem}
Let $I$ be a graph and $a, b\in V(I)$ be two adjacent vertices, if the degree of $a$ is one, then
\begin{center}
$^\partial\varepsilon_{I}(z)=^\partial\varepsilon_{I-a}(z)+(2z^2)^\partial\varepsilon_{I-a-b}(z)$.
\end{center}
\end{theorem}
\begin{proof}
We divide all the subsets of the vertex set of $I$ into four parts $V_1, V_2, V_3, V_4$, where
\begin{itemize}
  \item $V_1=\{A\subseteq V(I)|a\in A, b\in A\}$,
  \item $V_2=\{A\subseteq V(I)|a\notin A, b\notin A\}$,
  \item $V_3=\{A\subseteq V(I)|a\in A, b\notin A\}$,
  \item $V_4=\{A\subseteq V(I)|a\notin A, b\in A\}$.
\end{itemize}
According to Definition \ref{definition}, we have
\begin{center}
\begin{align*}
^\partial\varepsilon_{I}(z)&=\sum\limits_{i=1}^4\sum\limits_{A\in V_i}z^{\text{rank}(M_A)+\text{rank}(M_{A^c})} \\
&=2\sum\limits_{A\in V_2}z^{\text{rank}(M_A)+\text{rank}(M_{A^c})}+2\sum\limits_{A\in V_4}z^{\text{rank}(M_A)+\text{rank}(M_{A^c})}.
\end{align*}
\end{center}
Let us still use $A^c$ to denote the complement of $A$ in $V(I)$. For $^\partial\varepsilon_{I-a}(z)$, one computes
\begin{center}
\begin{align*}
^\partial\varepsilon_{I-a}(z)&=\sum\limits_{A\in V_2}z^{\text{rank}(M_A)+\text{rank}(M_{A^c-a})}+\sum\limits_{A\in V_4}z^{\text{rank}(M_A)+\text{rank}(M_{A^c-a})} \\
&=2\sum\limits_{A\in V_4}z^{\text{rank}(M_A)+\text{rank}(M_{A^c-a})} \\
&=2\sum\limits_{A\in V_4}z^{\text{rank}(M_A)+\text{rank}(M_{A^c})}.
\end{align*}
\end{center}
The second equation follows from the fact that, in $I-a$, there is a one-to-one correspondence between the elements in $V_2$ and that in $V_4$, and each two of them make the same contribution to the sum. The reason for the third equation is, since $A\in V_4$, it follows that $b\in A$ and $b\notin A^c$. Due to the assumption that $b$ is the only neighbour of $a$, the vertex $a$ is isolated in the subgraph induced by $A^c$, therefore rank$(M_{A^c-a})=\text{rank}(M_{A^c})$.

On the other hand, for $^\partial\varepsilon_{I-a-b}(z)$, we have
\begin{center}
\begin{align*}
^\partial\varepsilon_{I-a-b}(z)&=\sum\limits_{A\in V_2}z^{\text{rank}(M_A)+\text{rank}(M_{A^c-a-b})}\\
&=\sum\limits_{A\in V_2}z^{\text{rank}(M_A)+\text{rank}(M_{A^c})-2}.
\end{align*}
\end{center}
Here we offer a little explanation for the equation 
\begin{center}
$\text{rank}(M_{A^c-a-b})=\text{rank}(M_{A^c})-2$.
\end{center}
Actually, since $A\in V_2$, we know that $a, b\notin A$ and $a, b\in A^c$. Then we have
\begin{center}
$\text{rank}(M_{A^c})=\text{rank}\begin{pmatrix}
0 & 1 & \textbf{0} \\
1 & 0 & \ast \\
\textbf{0} & \ast & \ast
\end{pmatrix}=\text{rank}\begin{pmatrix}
0 & 1 & \textbf{0} \\
1 & 0 & \textbf{0} \\
\textbf{0} & \textbf{0} & \ast
\end{pmatrix}=\text{rank}(M_{A^c-a-b})+2$.
\end{center}

As a consequence, we obtain the desired result
\begin{center}
$^\partial\varepsilon_{I}(z)=^\partial\varepsilon_{I-a}(z)+(2z^2)^\partial\varepsilon_{I-a-b}(z)$.
\end{center}
The proof is finished.
\end{proof}

\subsection{The partial-dual genus polynomial and the skew characteristic polynomial}
Let $\mathcal{G}$ denote the vector space spanned by all graphs over the ground field $\mathbb{C}$. It can be equipped with a multiplication and a comultiplication, where the multiplication $m: \mathcal{G}\otimes\mathcal{G}\to\mathcal{G}$ sends two graphs into the disjoint union and the comultiplication $\Delta :\mathcal{G}\to\mathcal{G}\otimes\mathcal{G}$ is defined as $\Delta(I)=\sum\limits_{A\subseteq V(I)}I_A\otimes I_{A^c}$. As before, here $I_A$ and $I_{A^c}$ denote the subgraphs of $I$ induced by $A$ and $A^c$ respectively. It is well known that, by extending linearly, the multiplication and comultiplication make the vector space $\mathcal{G}$ into a commutative cocommutative bialgebra. In particular, both the multiplication and the comultiplication respect the four-term relation of graphs \cite{Lan2000}.

Recently, Dogra and Lando \cite{DL2023} introduced the \emph{skew characteristic polynomial} of graphs, which is defined as
\begin{center}
$Q_I(w)=\sum\limits_{A\subseteq V(I)}\delta_{\text{rank}(M_I), \text{rank}(M_A)}w^{|V(I)|-|A|}$,
\end{center}
here $\delta_{i, j}=1$ if $i=j$, otherwise $\delta_{i, j}=0$. A refined version, called the \emph{refined skew characteristic polynomial}, was defined as follows
\begin{center}
$\overline{Q}_I(w, z)=\sum\limits_{A\subseteq V(I)}w^{|V(I)|-|A|}z^{\text{rank}(M_A)}$.
\end{center}
It was proved that both the skew characteristic polynomial and the refined skew characteristic polynomial of simple graphs satisfy the 4-term relation of graphs \cite{DL2023}. 

\begin{remark}
Here we made a little modification on the definition of the refined skew characteristic polynomial. The original version of it given in \cite{DL2023} is defined as
\begin{center}
$\overline{Q}_I(w, z)=\sum\limits_{A\subseteq V(I)}w^{|V(I)|-|A|}z^{\text{corank}(M_A)}$.
\end{center}
Obviously, for a fixed graph $I$, these two definitions are essentially the same. On the other hand, since $I$, $I_{ab}'$ and $\tilde{I}_{ab}$ have the same number of vertices, if one satisfies the 4-term relaition, so is the other one.
\end{remark}

Motivated by the concept of graded partial-dual genus polynomial, which was categorified in \cite{CL2024}, we consider the following two-variable polynomial
\begin{center}
$^\partial\overline{\varepsilon}_{I}(w, z)=\sum\limits_{A\subseteq V(I)}w^{|V(I)|-|A|}z^{\text{rank}(M_A)+\text{rank}(M_{A^c})}$,
\end{center}
called the \emph{refined partial-dual genus polynomial} of graphs. Similar to the proof of Theorem \ref{theorem2}, one observes that the refined partial-dual genus polynomial also satisfies the 4-term relation of graphs. Essentially speaking, the reason why both $\overline{Q}_I(w, z)$ and $^\partial\overline{\varepsilon}_{I}(w, z)$ are 4-invariants is that the rank of the adjacency matrix is a 4-invariant \cite{Lan2000}. The definitions above suggest that the refined skew characteristic polynomial and the refined partial-dual genus polynomial are closely related.

Recall that an element $p$ in a bialgebra is \emph{primitive} if $\Delta(p)=1\otimes p+p\otimes1$. Denote the subspace of primitive elements of $\mathcal{G}$ by $\mathcal{P}$, then there exists an explicit projection $\pi: \mathcal{G}\to\mathcal{P}$, which is given by
\begin{center}
$\pi(I)=I-1!\sum\limits_{A_1\cup A_2=V(I)}I_{A_1}I_{A_2}+2!\sum\limits_{A_1\cup A_2\cup A_3=V(I)}I_{A_1}I_{A_2}I_{A_3}+\cdots$,
\end{center}
where each sum runs over the partitions of $V(I)$ into unordered nonempty subsets. The readers are suggested to refer to \cite{Lan1997,Lan2000} for more details. It was proved in \cite{DL2023} that $Q_{\pi(I)}(w)$ is a constant for any graph $I$ with at least two vertices. It is an interesting question to determine $^\partial\varepsilon_{\pi(I)}(z)$ for any graph $I$. On the other hand, consider the graph invariant $r_I(z)=z^{\text{rank}(M_I)}$, then the partial-dual genus polynomial $^\partial\varepsilon_I(z)$ is almost equal to twice the second term of $r_{\pi(I)}(z)$. This suggests us to investigate the following generalization of the partial-dual genus polynomial
\begin{center}
$^\partial\varepsilon_I^{(k)}(z)=\sum\limits_{A_1\cup\cdots\cup A_k=V(I)}z^{\sum\limits_{i=1}^k\text{rank}(M_{A_i})}$,
\end{center}
here $A_i\cap A_j=\emptyset$ if $i\neq j$ and $A_i\neq\emptyset$ for each $1\leq i\leq k$. It may help us better understand the value of $r_I(z)$ under the projection of $\pi$.

\section*{Acknowledgement}
The main result of this paper was finished when the author was visiting Zhejiang Normal University in November, 2023. We wish to thank Zixi Wang for the hospitality. The author is grateful to Sergei Chmutov for interesting discussions and Xian'an Jin, Qingying Deng for encouragement. The author was supported by the NSFC grant 12371065.

\bibliographystyle{amsplain}
\bibliography{}

\end{document}